\documentclass[11pt,letterpaper]{amsart}

\usepackage[margin=1.3in]{geometry}
\usepackage{url}
\usepackage[colorlinks=true, citecolor=blue]{hyperref}
\usepackage{bookmark}
    \usepackage{amssymb}
\usepackage{amsmath}
\usepackage{amscd,latexsym}
\usepackage{bookmark}
     \usepackage{amssymb,amsfonts,bm}
\usepackage[all,arc]{xy}
\usepackage{enumerate}
\usepackage{mathrsfs}
\usepackage{amscd}
\usepackage{tikz}
\usepackage{tikz-cd}
\usepackage{amsmath}
\usepackage{latexsym}
\usepackage{mathdots}
\usepackage{amsrefs}
\usepackage{graphicx}
\usepackage{mathtools}
\usepackage{leftidx}
\usepackage{tensor}
\usepackage{bbm}
\usepackage{dsfont}
\usepackage{enumitem}
\usepackage{xypic,amsthm, color}
\usepackage[new]{old-arrows}
\usepackage{enumitem,xcolor}
\usepackage{enumerate,mdwlist}
\usepackage{multicol}

\makeatletter
\@namedef{subjclassname@2020}{%
  \textup{2020} Mathematics Subject Classification}
\makeatother

\theoremstyle{plane}
\newtheorem{theorem}{Theorem}[section]
\newtheorem{lemma}[theorem]{Lemma}
 \newtheorem{corollary}[theorem]{Corollary}

        \newtheorem{theorem*}{Theorem}

    \newcounter{relctr} 
\everydisplay\expandafter{\the\everydisplay\setcounter{relctr}{0}} 

\AtBeginDocument{} 

\theoremstyle{definition}
 \newtheorem*{definition*}{Definition}

 \newtheorem{example}[theorem]{Example}

 \newtheorem*{data availability}{Data Availability}
 \newtheorem*{conflict}{Conflict of Interest}

\theoremstyle{remark}
\newtheorem{remark}[theorem]{Remark}
\newtheorem*{acknowledgments}{Acknowledgments}

\numberwithin{equation}{section}

\begin{document}
  \title[Euler Characteristics of Brill-Noether Loci on Prym Varieties]{Euler Characteristics of Brill-Noether Loci on Prym Varieties}
  \author{Minyoung Jeon}
\address{Department of mathematics, University of Georgia, Athens GA 30602, USA}
\email{minyoung.jeon@uga.edu}

   \begin{abstract}
In this article we compute formulas for the connected K-theory class of the pointed Brill-Noether loci in Prym varieties, which extend the result by De Concini and Pragacz. Applying the formulas, we compute the holomorphic Euler Characteristics of the loci.
 
         \end{abstract}



   \date{\today}


   \maketitle


\section{Introduction}

Historically, Leonhard Euler introduced the Euler characteristic for convex polyhedra in 1752, based on a paper about the notion of a graph in 1736.  His consideration for a graph from the geometry of position led to define the Euler characteristic for an arbitrary finite cell-complex. The Euler characteristic was furthermore generalized by Poincar\'{e} in the early 20th century, and turned out to be a topological invariant of a space. The Euler-Poincar\'{e} formula for the Euler characteristic of a topological space uses the so-called Betti numbers of the space. Since then, there has been a lot of variations of the formula and the most interesting one to consider in this paper is the holomorphic Euler characteristic of a sheaf $\mathcal{F}$ on a proper scheme $X$ which replaces the Betti numbers by the dimension of the $i$th cohomology group with coefficients in the sheaf. To be specific, 
\[
\chi(\mathcal{F})=\sum_i(-1)^ih^i(X;\mathcal{F}).
\]

Prym varieties named after Friedrich Prym are abelian varieties constructed from \'etale covers of algebraic curves. They had been investigated analytically by Schottky-Jung \cite{SJ}, Wirtinger \cite{Wi}, Farkas-Rauch \cite{FaRa}, and algebraically by Mumford \cite{Mum}. See \cite[\S1]{Farkas} for more precise details about the analytic and algebraic approaches. The study of Prym varieties has been active for decades. Inside Prym varieties, the Brill-Noether loci were constructed by Welters \cite{W85}.

The Euler characteristic of the Brill-Noether loci in Jacobians of non-singular general curves was computed in \citelist{\cite{EH}\cite{PP}} and generally with special vanishing at two marked points in \citelist{\cite{ACT21}\cite{CLPT}}. 
However, the Euler characteristics of the Brill-Noether loci for Prym varieties are less well known. Our goal of this article is to provide a formula for the Euler characteristic of the Brill-Noether loci in Prym varieties with prescribed vanishing orders at one point. 

The more precise statement of our main result is as follows. Let $\mathbb{K}$ be an algebraically closed field of characteristic not equal to $2$. Let $\pi:\widetilde{C}\rightarrow C$ be an \'etale double cover of a smooth algebraic curve $C$ of genus $g=g(C)$ over $\mathbb{K}$. We fix a point $P$ in $\widetilde{C}$. Given a sequence of integers
\[
\bold{a}=(0\leq a_0< a_1< \cdots<a_r\leq 2g-2),\]
we define the {\it pointed Brill-Noether locus of line bundles} $V_{\bold{a}}^r(P)$ in the Prym variety $\mathscr{P}^+$ for odd $r$ (or $\mathscr{P}^-$ for even $r$) as 
\begin{align*}
 V_{\bold{a}}^r(P):=\left\{L\in \mathscr{P}^\pm\;|\;h^0(\widetilde{C}, L(-a_{i}P)))\geq r+1-i\;\text{for all}\; i,\;  h^0(\widetilde{C},L)\equiv r+1\;\text{(mod 2)}\right\}.
 \end{align*}
The pointed Brill-Noether loci $V_{\bold{a}}^r(P)$ have the structure of the degeneracy loci of type D (See \S\ref{sec3}), implying its Cohen-Macaulayness under the conditions that $V_{\bold{a}}^r(P)$ is either empty or of codimension $\sum_i a_i$ in $\mathscr{P}^\pm$. To state the main theorem, we need some notations. For a sequence $\mathbf{a}$ as above, we define a partition
\[
\lambda_i=a_{r+1-i}\quad \text{for $i=1,\ldots,r+1$}.
\]
We denote by $\ell_\circ$ the number of non-zero parts of $\lambda_i$ in $\lambda$ and $|\mu|:=\sum\mu_i$ the sum of components $\mu_i$ of nonnegative integers.
Let $s_i=\ell_\circ-i-\lambda_{i}+1$ for $i=1,\ldots,\ell_\circ$ and define $h(\lambda,\bold{v},k)=|\lambda|+|\bold{v}|+k$ for $\lambda$ and any nonnegative sequence $\mathbf{v}$ and positive integer $k$ for convenience in notation. The unspecified notations in the statement below will be defined in Section \ref{sec5}. 

\begin{theorem}[Euler Characteristic]\label{thm:Euler}
Let $\pi:\widetilde{{C}}\rightarrow {C}$ be an \'{e}tale double cover of a smooth algebraic curve $C$ over $\mathbb{K}$. Let $P$ be a point in $\widetilde{C}$.  If $V^r_{\mathbf{a}}(P)$ has codimension $|\lambda|$ in $\mathscr{P}^\pm$, then the Euler characteristic $\chi\left(\mathcal{O}_{V^r_{\bold{a}}(P)}\right)$ is equal to
\begin{align*}
\sum_{\mathbf{u}, \mathbf{v}}\left(\prod_{i=1}^{\ell_\circ}(-1)^{u_i}\right.&\left.\binom{u_i+s_i}{v_i}\right)\\
\cdot&\left(\sum_{\sigma\in S_{\ell_\circ}}\sum_{k\geq0}\sum_{\bold{f}\in A_k^\sigma}\mathrm{sgn}(\sigma)\prod_{j=1}^{\ell_\circ/2}g_{f_{\sigma(2j-1),\sigma(2j)}}^{\sigma(2j-1),\sigma(2j)}\dfrac{h(\lambda,\bold{v},k)!}{2^{\ell_\circ/2-h({\lambda,\bold{v}},k)}(\ell_\circ/2)!}\right),
\end{align*}
for $\sigma\in S_{\ell_\circ+1}$ permuting $\{1,\ldots, \ell_\circ\}$ if $\ell_\circ$ is even, and
\begin{multline*}
\sum_{\mathbf{u}, \mathbf{v}}\left(\prod_{i=1}^{\ell_\circ}(-1)^{u_i}\right.\left.\binom{u_i+s_i}{v_i}\right)\\
\cdot\left(\sum_{\sigma\in S_{\ell_\circ+1}}\sum_{k\geq0}\sum_{\bold{f}\in A_k^\sigma}\mathrm{sgn}(\sigma)\prod_{j=1}^{(\ell_\circ+1)/2}g_{f_{\sigma(2j-1),\sigma(2j)}}^{\sigma(2j-1),\sigma(2j)}\dfrac{h(\lambda,\mathbf{v},k)!}{2^{(\ell_\circ+1)/2-h(\lambda,\mathbf{v},k)}((\ell_\circ+1)/2)!}\right).
\end{multline*}
for $\sigma\in S_{\ell_\circ+1}$ permuting $\{0,1,\ldots, \ell_\circ\}$ if $\ell_\circ$ is odd. Here, the sums are taken over nonnegative integer sequences $\mathbf{u}$ and $\mathbf{v}$ of length $\ell_\circ$.  
In particular, if $\lambda_i=\ell_\circ+1-i$ such that $s_i=0$, then we get the Euler characteristic $\chi(\mathcal{O}_{V^r})$ of the classical Brill-Noether loci $V^r$ in $\mathscr{P}^\pm$.
\end{theorem}
In the statement, $s_i$ can have a negative value, and we use the binomial coefficients for negative integers $-s$ defined by 
\[
\binom{-s}{t}=\dfrac{-s(-s-1)\cdots(-s-t+1)}{t!}.
\]

It is known that Chern class formulas for certain degeneracy loci of linear series can help us simplify computations in Brill-Noether theory. For instance Kempf and Laksov provided significantly simplified proofs for the existence theorem on the special divisors by studying the Porteous' formula \cite{KL}.
Moreover, the Euler characteristic of the two pointed Brill-Noether locus in Picard varieties can be obtained by applying the Chern class formula in K-theory for a Schubert variety associated to $321$-avoiding permutations in the flag bundle of Lie type A \cite{ACT21}. By definition the flag bundle $Fl(E)$ of type A on a nonsingular variety $X$ over an algebraically closed field $\mathbb{K}$ is a bundle of flags (or filtrations) of subspaces of $E$ from its vector bundle $E\rightarrow X$, so that it can be considered as $SL(n)/B$ where $B$ is a Borel subgroup of $SL(n)$ for some rank $n$. In fact, we can consider Prym varieties with special vanishings at one point as degeneracy loci having the structure of Schubert varieties of Lie type D in even orthogonal Grassmannians $OG(n,\mathbb{K}^{2n})=SO(2n)/P$ for the maximal parabolic subgroups $P$ of $SO(2n)$. 
 In this perspective we take the K-theoretic Chern class formulas for even orthogonal Grassmannian degeneracy loci of Lie type D to deduce Theorem \ref{thm:Euler}. The proof of Theorem \ref{thm:Euler} is attributed to the formula for the K-theory class of the pointed Brill-Noether loci in the Prym varieties. 

The following is our second main Theorem \ref{thm:m2} computing the connective K-theory classes for the pointed Brill-Noether loci $V_{\bold{a}}^r(P)$ in Prym varieties. The exact definition of the Pfaffian in the statement below will be introduced as \eqref{Pf} in Section \ref{sec4}.

\begin{theorem}[Theorem \ref{thm:con}, Connected ${\rm K}$-theory class]\label{thm:m2}
Assume that $V_{\bold{a}}^r(P)$ is of pure codimension $|\lambda|$. Then the class of $V_{\bold{a}}^r(P)$ is
\begin{align}\label{Form1}
[V_{\bold{a}}^r(P)]&=Pf^{\diamondsuit}_\lambda(d(1),\ldots,d(\ell_\circ);\beta)
\end{align}
in the connected $\mathrm{K}$-homology $CK_*(\mathscr{P}^\pm)$ of $\mathscr{P}^\pm$.
\end{theorem}

The connected algebraic K-theory for schemes introduced by Cai \cite{Cai} connects the Chow groups and Quillen's K-theory groups and later it is investigated by Dai and Levine \cite{DL} in motivic homotopy theory. We adapt a simpler version of the connective K-theory of a scheme: for nonsingular $X$, the connected K-homology of $X$ denoted by $CK_*(X)$ is a graded algebra over $\mathbb{Z}[\beta]$ so that it can be specialized to the Chow homology $A_*(X)$ at $\beta=0$ and to Grothendieck group $K_\circ(X)$ of coherent sheaves at $\beta=-1$. The reader might find \citelist{\cite{A}\cite{HIMN}\cite{HIMN20}} for the study of certain degeneracy loci in the context of the connected K-theory. In this regard our strategy to have Theorem \ref{thm:m2} is to apply the K-theoretic Chern class formulas for Grassmannian degeneracy loci of Lie type D in the notion of \cite[Section 4]{A}. 

As corollaries (Corollary \ref{Form1} and Corollary \ref{Form2}) of Theorem \ref{thm:m2}, the class of the Brill-Noether loci for Prym varieties in the Grothendiek group of $\mathscr{P}^\pm$ is given by 
\[
[\mathcal{O}_{V_{\bold{a}}^r(P)}]=Pf^{\diamondsuit}_\lambda(d(1),\ldots,d(r+1);-1)\in K_\circ(\mathscr{P}^\pm)
\]
by specializing at $\beta=-1$, and their singular cohomology $H^*(\mathscr{P}^\pm,\mathbb{C})$ with complex coefficients or numerical group $N^*(\mathscr{P}^\pm,\mathbb{K})$ with coefficients in an arbitrary field $\mathbb{K}$ of characteristic different than 2 at $\beta=0$ can be expressed by 
 \[
[V_{\bold{a}}^r(P)]= \dfrac{1}{2^{\ell_\circ}}\prod_{i=1}^{\ell_\circ}\dfrac{1}{{\lambda_i!}}\prod_{i<j}\dfrac{\lambda_i-\lambda_j}{\lambda_i+\lambda_j}\cdot (2\xi)^{|\lambda|}.
 \]
Here $\xi$ is the class of the theta divisor on $\mathscr{P}^\pm$ in $N^*(\mathscr{P}^\pm,\mathbb{K})$ or $H^*(\mathscr{P}^\pm,\mathbb{C})$. 

Indeed Theorem \ref{thm:m2} extends the formulas for the cohomology classes of the Brill-Noether loci of Prym varieties by Concini and Pragacz \cite{CP95}.
When it comes to $\beta=0$, we recover the class of the pointed Brill-Noether loci for Prym varieties, Corollary \ref{Form2} that coincides with the recent work of Tarasca \cite[Theorem 1]{Tarasca}.

Lastly our formulas are presented for the one-pointed case. As such, it would be interesting to investigate further the formulas for the Euler characteristics for the two-pointed Brill-Noether loci in Prym varieties. The author currently works on the subject in this direction.

The structure of this paper is the following. We review the classical Brill-Noether loci of Prym varieties in Chapter \ref{sec2} and K-theoretic class of even orthogonal degeneracy loci in Chapter \ref{sec3}. Our connected K-theory class formulas for the pointed Brill-Noether loci of Prym varieties are presented in Chapter \ref{sec4}. In the end we present the Euler characteristic class of the Brill-Noether loci in Prym varieties with special vanishing at one point in Chapter \ref{sec5}.

\section{Review on Brill-Noether loci of Prym Variety}\label{sec2}
 Let $C$ be a smooth algebraic curve of genus $g=g(C)$ over an algebraically closed field $\mathbb{K}$ whose characteristic is not equal to $2$. 
 Let $\pi:\widetilde{C}\rightarrow C$ be an \'etale double cover of $C$. 
 We denote by $J(C)$ and $J(\widetilde{C})$ the Jacobians of $C$ and $\widetilde{C}$ respectively. 
 We define a norm map $\mathrm{Nm}_{\pi}=\pi_*:\mathrm{Div}(\widetilde{C})\rightarrow \mathrm{Div}(C)$ \cite[Appendix
B]{ACGH} by sending a divisor $\sum q_i$ on $\widetilde{C}$ to the divisor $\sum\pi(q_i)$ on $C$. This map induces a map of Jacobians
$$\mathrm{Nm}_{\pi}:J(\widetilde{C})\rightarrow J(C).$$
Let $\tau:\widetilde{C}\rightarrow \widetilde{C}$ be the involution exchanging sheets of $\widetilde{C}$ over $C$. 
 We define the {\it Prym variety} $\mathscr{P}$ \cite{H82,M71} by
 $$\mathscr{P}=\mathrm{Ker}(id_{J(\widetilde{C})}+\tau)^0=\mathrm{Im}(id_{J(\widetilde{C})}-\tau)$$
 where $id_{J(\widetilde{C})}:J(\widetilde{C})\rightarrow J(\widetilde{C})$ is the identity map on $J(\widetilde{C})$ and the superscript $0$ implies the connected component containing the origin.

Since $J(C)$ can be identified with $\mathrm{Pic}^{2g-2}(C)$, and similarly for $\widetilde{C}$, the norm map can be regarded as a map $\mathrm{Nm}:\mathrm{Pic}^{2g-2}(\widetilde{C})\rightarrow \mathrm{Pic}^{2g-2}(C)$ of Picard groups. Let $K_C\in\mathrm{Pic}^{2g-2}(C)$ be the canonical divisor class. The inverse image of $K_C$ under the norm map is given by
 $$\mathrm{Nm}^{-1}(K_C)=\mathscr{P}^+\cup \mathscr{P}^-$$ where
 $\mathscr{P}^+=\{L:h^0(\widetilde{C},L)\equiv 0\;\text{(mod 2)}\}$ and $\mathscr{P}^-=\{L:h^0(\widetilde{C},L)\equiv 1\;\text{(mod 2)}\}$. 
 
The {\it Brill-Noether loci} in the Prym varieties $\mathscr{P}^\pm$ are set-theoretically defined by the closed subset
 \begin{align}\label{Brill}
 V^r=\{L\in\mathrm{Nm}^{-1}(K_C):h^0(\widetilde{C},L)\geq r+1, h^0(\widetilde{C},L)\equiv r+1\;\text{(mod 2)}\}\subset\mathrm{Pic}^{2g-2}(\widetilde{C})
 \end{align}
 
 \begin{example}
Let $C$ be a general curve of $g=2$ and $r=1$. Then the genus of $\widetilde{C}$ is $3$ such that 
$V^1\cong W_{2}^1(\widetilde{C})$ is $\mathbb{P}^1$-bundle of $J(C)$.
\end{example}

We describe the scheme and set-theoretical structure of $V^r$ used in \cite{CP95}. It is noteworthy that Welter \cite{W85} gave a different scheme structure on $V^r$, but the two scheme structures agree on an open dense subset. Interested readers may refer to comments in Introduction and Proposition 4 of \cite{CP95} for further details.


For the double cover $\pi:\widetilde{C}\rightarrow C$, let $1\times \pi:\mathrm{Pic}^{2g-2}(\widetilde{C})\times\widetilde{C}\rightarrow\mathrm{Pic}^{2g-2}(\widetilde{C})\times C$. We denote by $p:\mathrm{Pic}^{2g-2}(\widetilde{C})\times C\rightarrow\mathrm{Pic}^{2g-2}(\widetilde{C})$ and $q:\mathrm{Pic}^{2g-2}(\widetilde{C})\times C\rightarrow C$ the first and second projection. Let $\nu$ be the projection from $\mathrm{Pic}^{2g-2}(\widetilde{C})\times\widetilde{C}$ to $\mathrm{Pic}^{2g-2}(\widetilde{C})$. Then we have the following commutative diagram:

\begin{equation*}\label{Diag}
\begin{tikzcd}[column sep=normal]
 &\widetilde{C} \ar[r,"\pi"]&C\; &\\
\mathrm{Pic}^{2g-2}(\widetilde{C})\times\widetilde{C} \ar[ur, "\mu", bend left=10]\ar[drr, "\nu", bend right=10] \ar[r,"1\times \pi"] & \mathrm{Pic}^{2g-2}(\widetilde{C})\times C \ar[dr,"p"]\ar[ur, "q"] &&\\
&\;&\mathrm{Pic}^{2g-2}(\widetilde{C}) \ar[r, "\mathrm{Nm}"']&\mathrm{Pic}^{2g-2}(C)& \\
&&&
\end{tikzcd}\\
\end{equation*}

 For distinct fixed points $P_i$ on $C$, we consider a positive divisor $D=\sum_iP_i$ of $C$ for a sufficiently large $N=\mathrm{deg}(D)$, that is, $N+2g-2\geq 2\cdot g(\widetilde{C})+1$ or $N\geq2\cdot (2g-1)+1-(2g-2)=2g+1$ to have $h^0(K_C-D)=0$ by Riemann-Roch and Clifford's theorem.

Let $\mathcal{E}=(1\times \pi)_*\mathcal{L}$ of rank $2$. We write $\mathcal{E}(\pm D)$ for $\mathcal{E}\otimes q^*(\mathcal{O}_C(\pm D))$ where $D$ is a divisor. 
Let $\mathcal{V}=p_*(\mathcal{E}(D)/\mathcal{E}(-D))$, $\mathcal{U}=p_*(\mathcal{E}(D))$, and $\mathcal{W}=p_*(\mathcal{E}/\mathcal{E}(-D))$. Then $\mathcal{W}, \mathcal{U}\subset \mathcal{V}$ are subbundles of rank $n=2N$, since $\mathcal{W}$ is locally free of rank $n$ and $\mathcal{U}$ is just shifted by the divisor $-D$ from $p_*(\mathcal{E}(D)/\mathcal{E})$ where the rank of $p_*(\mathcal{E}(D)/\mathcal{E})$ is $2(N-(2g-2-(2g-1)+1))=n$. This enables $\mathcal{V}_{\mathscr{P}^\pm}$ to have a nondegenerate quadratic form with values in $\mathcal{O}_{\mathscr{P}^\pm}$ so that $\mathcal{W}_{\mathscr{P}^\pm}, \mathcal{U}_{\mathscr{P}^\pm}$ become maximal isotropic subbundles with respect to the form.

Specifically, let us fix $L\in \mathscr{P}^{\pm}$. For $E=\pi_*L$, let $V=H^0(C,E(D)/E(-D))$ of $2n$-dimensional vector space. On $V$ we define a symmetric form $Q:V\times V\rightarrow \mathbb{C}$ by
\begin{align*}
Q(\sigma,\tau)=\sum_i\mathrm{Res}_{P_i}(\sigma\tau)
\end{align*}
where $\sigma\tau\in H^0(C, L^2(2D)/L^2)=H^0(C, \omega_C(2D)/\omega_C)$ such that $Q$ is nondegenerate. 
In fact we can define $V$ as $V=U'\oplus W$ where $U'=H^0(C,E(D)/E)$ and $W=H^0(C,E/E(-D))$. The nondegenerate symmetric form $Q$ on $V$ is defined by
$$Q(\sigma_1\oplus \sigma_2,\tau_1\oplus \tau_2)=\sum\mathrm{Res}(\sigma_1\tau_2+\sigma_2\tau_1).$$

We can consider the symmetric form $Q$ as a quadratic form $\mathfrak{q}$ which sends $v$ to $Q(v,v)$ for $v\in V$. With the quadratic form on $V$ we see that $U'$ and $W$ are $n$-dimensional isotropic subspaces by the residue theorem. In other words, $U'$ consists of regular functions, which makes $U'$ an isotropic subspace, and for $W$, if $\sigma$ and $\tau$ are in $W$, then the sum of the residue of $\sigma\tau$ is zero, and so $W$ is also an isotropic subspace. Additionally we have another isotropic subspace $U=H^0(C,E(D))$ for the quadratic form by the restriction map of the space $H^0(C,E(D))$ in $V$. We notice that the intersection of $U$ and $W$ is global regular sections of $E$ and so $\mathrm{dim}(U\cap W)=h^0(C,E)$. 
Due to the choice of $\mathcal{L}$, the construction globalizes and thus defines set-theoretically the Brill-Noether loci (\ref{Brill}) on a Prym variety. Readers may refer to \cite{M71} for more details.

\section{The connected K-theory class of even orthogonal degeneracy loci}\label{sec3}

This section reviews on general formulas for even orthogonal degeneracy loci in the connective K-homology used later to find the class of the pointed Brill-Noether locus in a Prym variety in Section \ref{sec4}. To be precise, we use the connective K-homology with the natural isomorphisms between the Chow homology and the Grothendiek group of coherent sheaves 
$$CK_*(X)/(\beta=0)\cong A_*(X)\quad\text{and}\quad CK_*(X)/(\beta=-1)\cong K_\circ(X)$$ 
of nonsingular $X$ by specializing the parameter as $\beta=0$ and $\beta=-1$ respectively. Even if we choose $X$ to be singular, the isomorphism still works with their cohomologies via the operational cohomology theory. Moreover, for the closed subvariety $Y\subseteq X$, the class $\left[Y\right]\in A_*(X)$ and $\left[\mathcal{O}_Y\right]\in K_\circ(X)$ can be obtained from the fundamental class $\left[Y\right]\in CK_*(X)$.

 Let $\mathscr{V}\rightarrow X$ be a rank $2n$ vector bundle over a smooth irreducible algebraic variety $X$ over $\mathbb{K}$, equipped with a nondegenerate quadric form $\mathfrak{q}$. Let $\bold{OG}(n, \mathscr{V})$ be an orthogonal Grassmannian bundle with $\pi:\bold{OG}(n,\mathscr{V})\rightarrow X$.  We consider a rank $n$ tautological subbundle $\mathscr{S}$ of $\pi^*\mathscr{V}$ on $\bold{OG}(n,\mathscr{V})$. With a common abuse of notation, we may use $\mathscr{V}$ as $\pi^*\mathscr{V}$. Let $\mathscr{F}$ be a rank $n$ maximal isotropic subbundle of $\mathscr{V}$. Especially $\mathscr{S}$ is also a maximal isotropic subbundle of $\mathscr{V}$. Let us fix a flag of isotropic subbundles 
 $$\mathscr{F}_{p_r}\hookrightarrow \mathscr{F}_{p_{r-1}}\hookrightarrow\cdots\hookrightarrow \mathscr{F}_{p_{0}}\subseteq\mathscr{F}\xrightarrow{\phi}\mathscr{S}$$
with respect to the form on a variety $X$ where $\mathscr{F}_{p_i}$ has rank $n-p_i$ for all $i$. Especially $0\leq p_0<\cdots<p_{r}$. We define the degeneracy locus $V_{\mathbf{p}}^r$ associated to a sequence $\mathbf{p}=(0\leq p_0<\cdots<p_{r})$ to be 
 $$V_{\mathbf{p}}^{r}=\{x\in X|\;\mathrm{dim}(\mathscr{F}_{p_i}\cap \mathscr{S})_x\geq r+1-i,\quad \mathrm{dim}(\mathscr{F}\cap \mathscr{S})_x\equiv r+1 \;\text{(mod}\;2) \;\;\text{for }x\in X\}$$
for $0\leq i\leq r$. We remark that this degeneracy locus should be taken to be the closure of the locus where equality holds. It is known that $V_{\mathbf{p}}^r$ is a Cohen-Macaulay scheme if $V_{\mathbf{p}}^r$ is either empty or of codimension  $\sum_i p_i$ due to the main theorem in \cite{DCL} and the same reasoning as in the proof of \cite[Proposition 2 (1)]{CP95}.

Now, we slightly modify the Pfaffian formula in \cite[Theorem 4]{A19} to our setting. We define Euler classes $e(\mathscr{F}_{p_i},\mathscr{S})$ for isotropic subbundles $\mathscr{F}_{p_i}$ and $\mathscr{S}$. In other words, for maximal isotropic bundles $\mathscr{S}$ and $\mathscr{F}_{p_i}$, Euler classes are defined by
\begin{equation*}
e_m(\mathscr{F}_{p_i},\mathscr{S}) =
\begin{cases}
(-1)^{\mathrm{dim}(\mathscr{F}\cap \mathscr{S})}\gamma(\mathscr{F},\mathscr{S})c_{p_i}(\mathscr{F}/\mathscr{F}_{p_i}) & \text{if $m=p_i$}\\
0& \text{otherwise}
\end{cases}
\end{equation*}
where $c(\mathscr{F}/\mathscr{F}_{p_i})$ indicates the Chern class and $\gamma(\mathscr{S},\mathscr{F})$ in $CK^0(X)$ is the canonical square root of $c(\mathscr{V}-\mathscr{S}-\mathscr{F};\beta)$ \cite[Appendix B]{A}. 
We denote by $T_i$ the raising operator increasing the index of $c(i):=c(\mathscr{V}-\mathscr{S}-\mathscr{F}_{p_{r+1-i}})$ by one. Let $R_{ij}=T_i/T_j$ and $e(i):=e(\mathscr{F}_{p_{r+1-i}},\mathscr{S})$. Let $\lambda=(\lambda_{r+1}>\cdots>\lambda_0)$ be a strict partition defined by 
\[
\lambda_i=p_{r+1-i}\quad\text{for $i=0,\ldots,r+1$ }
\]
and denote by $\ell_\circ:=\ell(\lambda)$ the number of non-zero parts $\lambda_i$ of $\lambda$. Suppose that $\ell_\circ$ is even. 
We define the Pfaffian formula
\begin{equation}\label{eqn:pfa}
\mathrm{Pf}_\lambda(d(1),\ldots,d(\ell_\circ);\beta):=\mathrm{Pf}(M;\beta),
\end{equation}
where $d(i)=c(i)+(-1)^ie(i)$ for $1\leq i\leq \ell_\circ$, and the entry $m_{i,j}$ of the $\ell_\circ\times\ell_\circ$ skew-symmetric matrix $M$ is
\begin{align*}
m_{i,j}&=\dfrac{1-\delta_i\delta_jR_{ij}}{1+\delta_i\delta_j(R_{ij}-\beta T_i)}\cdot\dfrac{(1-\beta \widetilde{T}_i)^{{\ell_\circ}-i-\lambda_i+1}}{2-\beta\widetilde{T}_i}\cdot\dfrac{(1-\beta\widetilde{T}_j)^{{\ell_\circ}-j-\lambda_j+1}}{2-\beta\widetilde{T}_j}\\
&\cdot (c_{\lambda_i}(i)-(-1)^{{\ell_\circ}}e_{\lambda_i}(i))(c_{\lambda_j}(j)+(-1)^{{\ell_\circ}}e_{\lambda_j}(j)),
\end{align*}
with the skew-symmetric relations $m_{ji} =-m_{ij}$ and $m_{ii} = 0$. Here $\widetilde{T}_i=\delta_iT_i$ and $\delta_i$ assigns $(-1)^i$ to $0$ in $d(i)$. In particular, if $\ell_\circ$ is odd, we augment the matrix $M$ by the setting $m_{0j}=(1-\beta\widetilde{T}_j)^{{\ell_\circ}-j-\lambda_j+1}(2-\beta\widetilde{T}_j)^{-1}\cdot(c_{\lambda_j}(j)+e_{\lambda_j}(j))$ for $j=1,\ldots,\ell_\circ$.

Combined \cite[Theorem 4]{A19} with \cite[, Corrigendum Pg. 3]{A19} and then specialized to our setting, we have the formula for the class $\left[V_{\bold{p}}^r\right]$ as follows. 
\begin{theorem}\label{thy1}
Let $X$ be a variety. Then the connected K-theory class of $V^r_{\bold{p}}$ in $CK_*(X)\left[\frac{1}{2}\right]$ is given by
\begin{align*}\label{thy1}
\left[V^r_{\bold{p}}\right]&=\mathrm{Pf}_\lambda(d(1),\ldots,d(\ell_\circ);\beta)\cdot [X].
\end{align*}
\end{theorem}


\section{Classes of the pointed Brill-Noether loci on Prym variety}\label{sec4}

 \subsection{Brill-Noether Classes with a vanishing sequence}
 In this section we consider the class of the Brill-Noether loci in the Prym variety $\mathscr{P}^\pm$ with prescribed vanishing orders at one point. 
 
 Let $C$ be a smooth curve of genus $g$ and $\pi:\widetilde{C}\rightarrow C$ be an \'etale double cover of $C$. For a line bundle $L$ in $ V^r$, 
the vanishing sequence at $P\in \widetilde{C}$ is given by the sequence
\begin{equation*}
\bold{a}(P)=(0\leq a_0^L(P)<\cdots<a_r^L(P)\leq 2g-2)
\end{equation*}
of vanishing orders in the $2n$-dimensional vector space $\mathcal{V}_{\mathscr{P}^\pm}=p_*(\mathcal{E}(D)/\mathcal{E}(-D))_{\mathscr{P}^\pm}$ at $P$ such that it is the maximal sequence satisfying the condition 
\[
h^0(\widetilde{C}, L(-a_{i}^L(P)\cdot P)))\geq r+1-i\;\text{for all}\; i. 
\]
The reader may refer to the work of Eisenbud-Harris used in \cite{W85} for the vanishing sequence.
Let us fix a point $P$ in $\widetilde{C}$ and the sequence 
\[
\bold{a}=(0\leq a_0< a_1< \cdots<a_r\leq 2g-2).\]
 We define the {\it pointed Brill-Noether loci of line bundles} $V_{\bold{a}}^r(P)$ in the Prym variety $\mathscr{P}^+$ for odd $r$ (or $\mathscr{P}^-$ for even $r$) by  \begin{align*}
 V_{\bold{a}}^r(P):=\left\{L\in \mathrm{Nm}^{-1}(K_C)\;|\;\right.&h^0(\widetilde{C}, L(-a_{i}P)))\geq r+1-i\;\text{for all}\; i \\
&\left. h^0(\widetilde{C},L)\equiv r+1\;\text{(mod 2)}\right\}\subset \mathrm{Pic}^{2g-2}(\widetilde{C}).
 \end{align*}
  The structure of this variety can be considered by an even orthogonal Grassmannian degeneracy locus $V_{\bold{a}}^r(P)$ in $\mathrm{Pic}^{2g-2}(\widetilde{C})$ as in \S\ref{sec3}. In particular this construction generalizes the result for the Brill-Noether loci without vanishing orders in the Prym variety introduced in \S\ref{sec2}.

We recall that $\mathcal{L}$ is the Poincar\'e line bundle on $\mathrm{Pic}^{2g-2}(\widetilde{C})\times \widetilde{C}$, $\mathcal{E}=(1\times \pi)_*\mathcal{L}$ and $D=\sum_{i} P_i$ is a divisor for distinct fixed points $P_i$ on $C$. Assume that $\pi(P)\neq P_i$ for all $i$. Let $\widetilde{D}=\pi^*(D)$. We note that $p_*(\mathcal{E}\otimes q^*(\mathcal{O}_C(D)))=\nu_*(\mathcal{L}\otimes\mu^*\mathcal{O}_{\widetilde{C}}(\widetilde{D}))=\nu_*(\mathcal{L}(\widetilde{D}))$. 

Now, we set
\begin{align*}
\mathcal{W}_i:&=\nu_*(\mathcal{L}\otimes \mu^*(\mathcal{O}_{\widetilde{C}}(\widetilde{D}-a_{i}P)))=\nu_*(\mathcal{L}(\widetilde{D}-a_{i}P))
\end{align*}
for $0\leq i\leq r$.
Then the sheaf $\mathcal{W}_i$ is the vector bundle of rank
\[
\mathrm{rk}(\mathcal{W}_i)=n-a_{i}
\] 
so that we have a filtration $\mathcal{W}_{r}\subset \mathcal{W}_{r-1}\subset\cdots\subset \mathcal{W}_0\subseteq\mathcal{W}:=p_*(\mathcal{E}(D))=\nu_*(\mathcal{L}(\widetilde{D}))$. Then there is a natural sequence
 \[
 (\mathcal{W}_{r})_{\mathscr{P}^\pm}\hookrightarrow (\mathcal{W}_{r-1})_{\mathscr{P}^\pm}\hookrightarrow \cdots\hookrightarrow (\mathcal{W}_0)_{\mathscr{P}^\pm}\subseteq\mathcal{W}_{\mathscr{P}^\pm}
  \]
 of vector bundles on $\mathscr{P}^\pm$. Here the nondegenerate symmetric form $Q$ defined in \S\ref{sec2} is naturally inherited to their subbunddles. In addition, for $L\in \mathscr{P}^{\pm}$, if we set $U=H^0(\widetilde{C},L/L(-\widetilde{D}))$ and $W_i=H^0(\widetilde{C},L(\widetilde{D}-a_i(P)\cdot P))$, $U\cap W_i$ is global regular sections of $L$ such that 
 $$\mathrm{dim}(U\cap W_i)=h^0(\widetilde{C},L(-a_iP)).$$
Hence $V_{\bold{a}}^r(P)$ can be regarded as the degeneracy loci in \S\ref{sec3}, so that it is a Cohen-Macaulay scheme provided that $V_{\bold{a}}^r(P)$ is either empty or of codimension  $\sum_i a_i$.

Let $\lambda$ be the partition associated to the vanishing orders $\mathbf{a}$:
\[
\lambda_i:=a_{r+1-i}\quad\text{for}\;i=1,\ldots, r+1.
\] 
We recall that $\ell_\circ:=\ell(\lambda)$ is the number of non-zero parts of $\lambda_i$ of $\lambda$. 

Since $\mathcal{U}_{\mathscr{P}^\pm}$ has the trivial Chern class, the Chern classes $c(i)=c(\mathcal{V}_{\mathscr{P}^\pm}-\mathcal{U}_{\mathscr{P}^\pm}-(\mathcal{W}_{r+1-i})_{\mathscr{P}^\pm})$ become $c(\mathcal{V}_{\mathscr{P}^\pm}-(\mathcal{W}_{r+1-i})_{\mathscr{P}^\pm})=c((\mathcal{W}_{r+1-i})_{\mathscr{P}^\pm}^\vee)$, and the Euler classes $e_j(i)=e_j((\mathcal{W}_{r+1-i})_{\mathscr{P}^\pm},\mathcal{U}_{\mathscr{P}^\pm})$ are equal to 
\[
e_j(i)=(-1)^{\mathrm{dim}(\mathcal{U}\cap\mathcal{W})}\gamma(\mathcal{W}_{\mathscr{P}^\pm},\mathcal{U}_{\mathscr{P}^\pm}) c_j(\mathcal{W}_{\mathscr{P}^\pm}/(\mathcal{W}_{r+1-i})_{\mathscr{P}^\pm})
\] if $j=\lambda_i$ and $0$ otherwise. Here $\gamma(\mathcal{W}_{\mathscr{P}^\pm},\mathcal{U}_{\mathscr{P}^\pm})$ is the canonical square root of the CK-theoretic Chern class $c((\mathcal{W}_{i})_{\mathscr{P}^\pm}^\vee;\beta)$. Therefore, with these specializations for $d(i)=c(i)+(-1)^ie(i)$ for $1\leq i\leq \ell_\circ$, and the Pfaffian formula  \eqref{eqn:pfa} yelds 
\begin{align}\label{Pf}
Pf^{\diamondsuit}_\lambda(d(1),\ldots,d(\ell_\circ);\beta)&=Pf(M^{\diamondsuit};\beta)
\end{align}
of the $\ell_\circ\times\ell_\circ$ skew-symmetric matrix $M^{\diamondsuit}$ whose entries $m^{\diamondsuit}_{i,j}$ are given by 
\begin{align*}
m^{\diamondsuit}_{i,j}&=\dfrac{1-\delta_i\delta_jR_{ij}}{1+\delta_i\delta_j(R_{ij}-\beta T_i)}\cdot\dfrac{(1-\beta \widetilde{T}_i)^{{\ell_\circ}-i-\lambda_i+1}}{2-\beta\widetilde{T}_i}\cdot\dfrac{(1-\beta\widetilde{T}_j)^{{\ell_\circ}-j-\lambda_j+1}}{2-\beta\widetilde{T}_j}\\
&\cdot (c_{\lambda_i}((\mathcal{W}_{r+1-i})_{\mathscr{P}^\pm}^\vee)-(-1)^{{\ell_\circ}+\mathrm{dim}(\mathcal{U}\cap\mathcal{W})}\gamma(\mathcal{W}_{\mathscr{P}^\pm},\mathcal{U}_{\mathscr{P}^\pm}) c_{\lambda_i}(\mathcal{W}_{\mathscr{P}^\pm}/(\mathcal{W}_{r+1-i})_{\mathscr{P}^\pm})\\
&\cdot(c_{\lambda_j}((\mathcal{W}_{r+1-j})_{\mathscr{P}^\pm}^\vee)+(-1)^{{\ell_\circ}+\mathrm{dim}(\mathcal{U}\cap\mathcal{W})}\gamma(\mathcal{W}_{\mathscr{P}^\pm},\mathcal{U}_{\mathscr{P}^\pm}) c_{\lambda_j}(\mathcal{W}_{\mathscr{P}^\pm}/(\mathcal{W}_{r+1-j})_{\mathscr{P}^\pm}).
\end{align*}
In case that $\ell_\circ$ is odd, we use the augmented matrix $M^{\diamondsuit}$ by putting 
\begin{multline}
\label{aug}
m^{\diamondsuit}_{0j}=(1-\beta\widetilde{T}_j)^{{\ell_\circ}-j-\lambda_j+1}(2-\beta\widetilde{T}_j)^{-1}\\
\cdot(c_{\lambda_j}((\mathcal{W}_{r+1-j})_{\mathscr{P}^\pm}^\vee)+(-1)^{\mathrm{dim}(\mathcal{U}\cap\mathcal{W})}\gamma(\mathcal{W}_{\mathscr{P}^\pm},\mathcal{U}_{\mathscr{P}^\pm}) c_{\lambda_j}(\mathcal{W}_{\mathscr{P}^\pm}/(\mathcal{W}_{r+1-j})_{\mathscr{P}^\pm})
\end{multline}
for $j=1,\ldots,\ell_\circ$. Here $Pf^{\diamondsuit}_\lambda(d(1),\ldots,d(\ell_\circ);\beta)$ and $M^{\diamondsuit}$ are the specialized ones with our specific bundles $\mathcal{W}_\bullet$ and $\mathcal{U}$ from $Pf_\lambda(d(1),\ldots,d(\ell_\circ);\beta)$ and $M$ used in \eqref{eqn:pfa}. Then by Theorem \ref{thy1}, we have the connected K-theory class of $V_{\bold{a}}^r(P)$ as follows.

\begin{theorem}\label{thm:con}
Assume that $V_{\bold{a}}^r(P)$ has pure codimension $|\lambda|:=\sum_i\lambda_i$. Then we have the class of $V_{\bold{a}}^r(P)$ given by 
\[
[V_{\bold{a}}^r(P)]=Pf^{\diamondsuit}_\lambda(d(1),\ldots,d(\ell_\circ);\beta)\in CK_*(\mathscr{P}^\pm).
\]
\end{theorem}

 Specializing at $\beta=-1$, we have the K-theory class of the pointed Brill-Noether locus $V_{\bold{a}}^r(P)$ in the Grothendieck group of coherent sheaves $K_\circ(\mathscr{P}^\pm)$ for $\mathscr{P}^\pm$ as below.
\begin{corollary}[$\beta=-1$]\label{Form1}
Assume that $V_{\bold{a}}^r(P)$ has pure codimension $|\lambda|$. Then the K-theory class of $V_{\bold{a}}^r(P)$ is
\[
[\mathcal{O}_{V_{\bold{a}}^r(P)}]=Pf^{\diamondsuit}_\lambda(d(1),\ldots,d(r+1);-1)\in K_\circ(\mathscr{P}^\pm).
\]
\end{corollary}

Similarly, with a specialization at $\beta=0$, we obtain the class of $V_{\bold{a}}^r(P)$ in Chow homology $A_*(\mathscr{P}^\pm)$ of $\mathscr{P}^\pm$ as follows.
\begin{corollary}[$\beta=0$]\label{cor2}
Assume that $V_{\bold{a}}^r(P)$ has pure codimension $|\lambda|$. Then the class of $V_{\bold{a}}^r(P)$ is
\[
[{V_{\bold{a}}^r(P)}]=Pf^{\diamondsuit}_\lambda(d(1),\ldots,d(r+1);0)\in A_*(\mathscr{P}^\pm).
\]
\end{corollary}

Furthermore, we can find the numerical equivalence class of $V_{\bold{a}}^r(P)$ in the numerical group $N^*(\mathscr{P}^\pm,\mathbb{K})$ with coefficients in $\mathbb{K}$ or its cohomology class of $V_{\bold{a}}^r(P)$ in the singular cohomology $H^*(\mathscr{P}^\pm,\mathbb{C})$ with complex coefficients, as a corollary (Corollary \ref{Form2}) of Theorem \ref{thm:con}. Here $k$ is an arbitrary field of characteristic not equals to $2$. 

In fact, $d_j(i)$ can be specialized in those cohomology rings with $\beta=0$ as the following lemma. The rest of this section must be read at $\beta=0$. Let ${\Theta}$ be the theta divisor on $\mathrm{Pic}^{2g-2}(\widetilde{C})$. We use the convention that the cohomology or numerical equivalence class $\theta$ of $\Theta$ does not rely on the choice of the divisors $\widetilde{D}-a_iP$ assumed in \cite[318-319]{ACGH}. 

\begin{lemma}\label{Chrn}
Given sufficiently positive divisor $\widetilde{D}-a_iP$, we have
\begin{equation*}
d_j(i)=\dfrac{(\theta')^j}{j!}
\end{equation*}
where $\theta'$ is the cohomology class restricted to $\mathscr{P}^\pm$ of the class of the theta divisor $\Theta$ on $\mathrm{Pic}^{2g-2}(\widetilde{C})$.
\end{lemma}
\begin{proof}
 According to modulo numerical equivalence (as in \cite[Equation (4)]{Ma}) and using the Poincar\'e dual formula, the Chern class of $\mathcal{W}_i^\vee$ is given by 
 \begin{align}\label{eqn:chern}
 c_j((\mathcal{W}_i)_{\mathscr{P}^\pm}^\vee)=\dfrac{(\theta')^j}{j!}.
 \end{align}
 Since $c(\mathcal{W})=e^{-\theta'}$ and $c_i(\mathcal{U}_{\mathscr{P}^\pm})=0$ for all $i>0$ by \cite[Lemma 5]{CP95}, we get
\begin{align*}
d_j(i)=c_j((\mathcal{W}_{i})_{\mathscr{P}^\pm}^\vee)+(-1)^i\cdot(-1)^{\mathrm{dim}(\mathcal{U}\cap\mathcal{W})}\cdot c_j(\mathcal{W}_{\mathscr{P}^\pm}/(\mathcal{W}_{i})_{\mathscr{P}^\pm})
\end{align*}
for $j=\lambda_i$. In case of $j\neq \lambda_i$, $c_j(i)=c_j((\mathcal{W}_{i})_{\mathscr{P}^\pm}^\vee).$ 
Let us take $j=\lambda_i$. Then the $j$-th Chern class of $\mathcal{W}_{\mathscr{P}^\pm}/(\mathcal{W}_{i})_{\mathscr{P}^\pm}$ vanishes as 
\begin{align*}
c_j(\mathcal{W}_{\mathscr{P}^\pm}/(\mathcal{W}_{i})_{\mathscr{P}^\pm})&=\sum_{k=0}^{j}c_k(\mathcal{W}_{\mathscr{P}^\pm})\cdot c_{j-k}((\mathcal{W}_{i})_{\mathscr{P}^\pm}^\vee)\\
&=\sum_{k=0}^j \dfrac{(-1)^k\cdot(\theta')^k}{k!}\cdot\dfrac{(\theta')^{j-k}}{(j-k)!}=\sum_{k=0}^j\dfrac{(-1)^k}{k!}\cdot\dfrac{1}{(j-k)!}(\theta')^j\\
&=\dfrac{(\theta')^j}{j!}\sum_{k=0}^j\dfrac{(-1)^k\cdot j!}{k!\cdot(j-k)!}=\dfrac{(\theta')^j}{j!}\sum_{k=0}^j(-1)^k\cdot\binom{j}{k}\\
&=\dfrac{(\theta')^j}{j!}\cdot ((-1)+1)^j=0.
\end{align*}
Hence we arrive at 
\begin{equation*}
d_j(i)=\dfrac{(\theta')^j}{j!}
\end{equation*}
as the $j$th degree of $e^{\theta'}$. 
Since $d_j(i)$ is a multiple of $(\theta')^j$, the class $\left[V_{\mathbf{a}}^r(P)\right]$ can be expressed by $\gamma\cdot(\theta')^{|\lambda|}$ for a rational number $\gamma$ and $|\lambda|=\sum_{i=1}^{r+1}\lambda_i$.
\end{proof}

Let $\xi$ be the class of the theta divisor on $\mathscr{P}^\pm$ in $N^*(\mathscr{P}^\pm,\mathbb{K})$ or $H^*(\mathscr{P}^\pm,\mathbb{C})$. Then we have $\theta'=2\xi$ via \cite[p.342]{Mum}. By the above lemma, the class of $V_{\bold{a}}^r(P)$ becomes:

\begin{corollary}[$\beta=0$]\label{Form2}
Suppose that $V_{\bold{a}}^r(P)$ has the dimension equal to $\rho=g-1-|\lambda|$. Then the class of the pointed Brill-Noether loci for $\mathscr{P}^\pm$ is
\begin{align}
[V_{\bold{a}}^r(P)]&=\dfrac{1}{2^{\ell_\circ}}\prod_{i=1}^{\ell_\circ}\dfrac{1}{{\lambda_i!}}\prod_{i<j}\dfrac{\lambda_i-\lambda_j}{\lambda_i+\lambda_j}\cdot (2\xi)^{|\lambda|}
\end{align}
in either $N^*(\mathscr{P}^\pm,\mathbb{K})$ or $H^*(\mathscr{P}^\pm,\mathbb{C})$
\end{corollary}

\begin{proof}
We know from Corollary \ref{cor2} that 
\begin{align*}
[V_{\bold{a}}^r(P)]&=Pf^{\diamondsuit}_\lambda(d(1),\ldots,d(\ell_\circ);0)\in A_*(\mathscr{P}^\pm).
\end{align*}
We assume $\ell_\circ$ even. With Lemma \ref{Chrn}, the right hand side becomes the Pfaffian of the $\ell_\circ\times \ell_\circ$ skew-symmetric matrix $(m^\dagger_{ij})$ for $\ell_\circ\geq j>i\geq 1$ where  
\begin{align}\label{Mij}
m^\dagger_{ij}&=\dfrac{1}{2^2}\dfrac{{(2\xi)}^{\lambda_i+\lambda_j}}{(\lambda_i+\lambda_j)!}\left(\binom{\lambda_i+\lambda_j}{\lambda_i}+2\sum_{u>0}(-1)^{u}\cdot\binom{\lambda_i+\lambda_j}{\lambda_i+u}\right),
\end{align}
with the skew-symmetric relations $m^\dagger_{ji} =-m^\dagger_{ij}$ and $m^\dagger_{ii} = 0$.
If $\ell_\circ$ is odd, the matrix is augmented by $m^\dagger_{0j}=d_{\lambda_j}$ for $j=1,\ldots,\ell_\circ$.

 Since 
 \[
 \sum_{k=0}^{\lambda_i}(-1)^k\binom{\lambda_i+\lambda_j}{k}+\sum_{u>0}^{\lambda_j}(-1)^{u+\lambda_i}\binom{\lambda_i+\lambda_j}{\lambda_i+u}=\sum_{j=0}^{\lambda_i+\lambda_j}(-1)^j\binom{\lambda_i+\lambda_j}{j}=0,
\] 
 we have
 \begin{align*}
 (-1)^{\lambda_i}\sum_{u>0}^{\lambda_j}(-1)^u\binom{\lambda_i+\lambda_j}{\lambda_i+u}&=-\sum_{k=0}^{\lambda_i}(-1)^k\binom{\lambda_i+\lambda_j}{k}=(-1)^{\lambda_i+1}\binom{\lambda_i+\lambda_j-1}{\lambda_i}.
 \end{align*}
 By canceling $(-1)^{\lambda_i}$, we get to 
 \begin{equation}\label{eqn:id}
 \sum_{u>0}^{\lambda_j}(-1)^u\binom{\lambda_i+\lambda_j}{\lambda_i+u}=-\binom{\lambda_i+\lambda_j-1}{\lambda_i}.
 \end{equation}
Plugging equation \eqref{eqn:id} to \eqref{Mij} gives 
\begin{align*}
m^\dagger_{ij}&=\dfrac{1}{2^2}\dfrac{{(2\xi)}^{\lambda_i+\lambda_j}}{\lambda_i!\lambda_j!}\left(\dfrac{\lambda_i-\lambda_j}{\lambda_i+\lambda_j}\right).
\end{align*}
Using \cite[Appendix D]{FP}, 
\[
\mathrm{Pf}\left(\dfrac{1}{2^2}\dfrac{{(2\xi)}^{\lambda_i+\lambda_j}}{\lambda_i!\lambda_j!}\cdot \dfrac{\lambda_i-\lambda_j}{\lambda_i+\lambda_j}\right)
=\dfrac{1}{2^{\ell_\circ}}\prod_{i=1}^{\ell_\circ}\dfrac{1}{{\lambda_i!}}\prod_{i<j}\dfrac{\lambda_i-\lambda_j}{\lambda_i+\lambda_j}(2\xi)^{|\lambda|}.
\qedhere
\]
\end{proof}

While working this paper, the author learned from private conversation with David Anderson that Corollary \ref{Form2} was be found independently in \cite[Theorem 1]{Tarasca}. To be rigorous, we take this occasion to provide a more detailed proof, as
the proof of \cite[Theorem 1]{Tarasca} was sketched.

\begin{remark}
When $\mathbf{a}=(0,\ldots,r)$, we can recover the formula for the Brill-Noether loci in Prym varieties \cite[Theorem 9]{CP95} with imposed vanishing orders $\mathbf{a}$ at a point $P\in \widetilde{C}$. In this case, if $r$ is even, we take a strict partition $\lambda=(r,\ldots,1)$ which is often denoted by $\rho_r$ in many literature including in the proof of \cite[Lemma 8]{CP95} and \cite[Introduction, pg 14]{PR}. Then by Theorem \ref{Form2} we have
\begin{equation}\label{a}
[V_{\bold{a}}^r(P)]=\dfrac{1}{2^{r}}\prod_{i=1}^{r}\dfrac{1}{{i!}}\prod_{j<i}\dfrac{i-j}{i+j}\cdot (2\xi)^{r(r+1)/2},
\end{equation}
and if $r$ is odd, we set $\lambda=(r,\ldots,1,0)$ by putting $\lambda_{r+1}=a_0$ to be zero. Then Theorem \ref{Form2} gives
\begin{equation}\label{b}
[V_{\bold{a}}^r(P)]=\dfrac{1}{2^{r}}\prod_{i=0}^{r}\dfrac{1}{{i!}}\prod_{j<i}\dfrac{i-j}{i+j}\cdot (2\xi)^{r(r+1)/2}.
\end{equation}
We note that \eqref{a} agrees with \eqref{b}, since ${1}/{0!}=1$ and $\prod_{0<i}\dfrac{i-0}{i+0}=1$.
\end{remark}

\section{Euler Characteristics}\label{sec5}
In this section we provide formulas for the Euler characteristic of the pointed Brill-Noether loci $V^r_{\bold{a}}(P)$ in the Prym variety $\mathscr{P}^{\pm}$ associated to a fixed sequence 
\[
\bold{a}=(0\leq a_0<\cdots<a_r\leq 2g-2).
\]
We employ Hirzebruch-Riemann-Roch to find the Euler characteristic of the Brill-Noether loci for Prym varieties. Since the Todd class of $\mathrm{Pic}^{2g-2}(\widetilde{\mathcal{{C}}})$ is trivial, we have 
\[\chi\left(\mathcal{O}_{V^r_{\bold{a}}(P)}\right)=\int_{\mathrm{Pic}^{2g-2}(\widetilde{{C}})}\mathrm{ch}\left(\left[\mathcal{O}_{V^r_{\bold{a}}(P)}\right]\right).\]
To compute the Euler characteristic, we use the following lemma showing that cohomology Chern classes coincide with K-theory Chern classes via the Chern character isomorphism.

\begin{lemma}[\cite{ACT21}]\label{lem:char}
For a rank $n$ vector bundle $E$, if $ch(E)_i=0$ for $i>1$, then $ch(c_i^K(E))=c_i(E),$ where $c^K$ is the K-theory Chern classes. 
\end{lemma}
It is from Lemma \ref{lem:char} and \eqref{eqn:chern} that
\[
 \mathrm{ch}\left(c^K_j((\mathcal{W}_i\right)_{\mathscr{P}^\pm}^\vee))=\dfrac{(\theta')^j}{j!}
 \]
and thus
\[
 \mathrm{ch}\left(d^K_j(i)\right)=\dfrac{(\theta')^j}{j!}.
\]

The rest of this section is devoted to the proof of our main Theorem \ref{thm:Euler}. Lemma \ref{lem:char} and Corollary \ref{Form1} have a major role in obtaining the formula for the Euler characteristic. Before we actually prove the theorem, further notations must be introduced. Given any nonnegative integer sequence $\mathbf{v}=\{v_i\}_{i=1}^{\ell_\circ}$, we define $\left\{g_m^{i,j}\right\}_{m\geq 0}$ by
\[
g_0^{i,j}=\dfrac{1}{(\lambda_i+v_i)!(\lambda_j+v_j)!}+\sum_{\ell>0}(-1)^\ell\dfrac{2}{(\lambda_i+\ell+v_i)!(\lambda_j-\ell+v_j)!}
\]
and
\begin{align*}
g_m^{i,j}&=(-1)^m\left(\dfrac{1}{(\lambda_i+m+v_i)!(\lambda_j+v_j)!}\right.\\
&\left.+\sum_{\ell>0}(-1)^\ell\left(\binom{\ell+m-1}{m}+\binom{\ell+m}{m}\right)\dfrac{1}{(\lambda_i+\ell+m+v_i)!(\lambda_j-\ell+v_j)!}\right)\;\text{for}\; m>0\\
\end{align*}
 for $\ell_\circ \geq j > i \geq 1$, and $g_m^{ii}=0$, $g_m^{ji}=-g_m^{ij}$. We denote by $S_{2n}$ the symmetric group of degree $2n$. For $\sigma\in S_{2n}$, we define $\hat{f}(\sigma):=\sum_{j=1}^{n}f_{\sigma(2j-1),\sigma(2j)}$ for a nonnegative double sequence $\mathbf{f}:=\{f_{i,j}\}_{1\leq i,j\leq 2n}$, and let $A_i^\sigma=\left\{\bold{f}\;|\;\hat{f}(\sigma)=i\right\}$. In particular, $f_{0,j}=0=f_{j,0}$ and $g_{0}^{0,j}=\dfrac{1}{(\lambda_j+v_j)!}=-g_{0}^{j,0}$ for all $j=1,\ldots,\ell_\circ$ if necessary. 


\begin{proof}[Proof of Theorem \ref{thm:Euler}]
Let $P$ be a point on $\widetilde{C}$. We first compute $\mathrm{ch}\left(\left[\mathcal{O}_{V^r_{\mathbf{a}}(P)}\right]\right)$. We know from Corollary \ref{Form1} that
\[
[\mathcal{O}_{V_{\bold{a}}^r(P)}]=Pf^{\diamondsuit}_\lambda(d(1),\ldots,d(\ell_\circ);-1)\in K_\circ(\mathscr{P}^\pm).
\]
Here we assume that $\ell_\circ$ is even throughout this proof unless specified. If $\ell_\circ$ is odd, we use the augmented matrix for the Pfaffian formula with \eqref{aug} evaluated at $\beta=-1$ as defined. 

Then by the virtue of Lemma \ref{lem:char}, taking $[\mathcal{O}_{V_{\bold{a}}^r(P)}]$ under the Chern character isomorphism replaces the K-theory Chern classes involved in the Pfaffian formula by cohomology Chern classes. 

According to Lemma \ref{Chrn} and with the specialization at $\beta=-1$, $\mathrm{ch}\left(\left[\mathcal{O}_{V^r_{\mathbf{a}}(P)}\right]\right)$ is equal to the Pfaffian $\mathrm{Pf}(M^\star)$ where $M^\star=(m^\star_{ij})$ is the skew-symmetric matrix whose entries are
\begin{align}\label{Mij}
m^\star_{ij}&=\dfrac{1-R_{ij}}{1+R_{ij}+ T_i}\cdot \dfrac{(1+ T_i)^{{\ell_\circ}-i-\lambda_{i}+1}}{2+ T_i}\cdot \dfrac{(1+ T_j)^{{\ell_\circ}-j-\lambda_j+1}}{2+ T_j}\cdot d_{\lambda_i}(i)d_{\lambda_j}(j)
\end{align}
 for $\ell_\circ\geq j>i\geq 1$ with relations $m^\star_{ji} =-m^\star_{ij}$ and $m^\star_{ii} = 0$. As in \cite[Pg. 460]{A}, we can unfold $m^\star_{ij}$ by writing
\begin{equation}\label{com}
m^\star_{ij}=\dfrac{(1+ T_i)^{{\ell_\circ}-i-\lambda_{i}+1}}{2+ T_i}\cdot \dfrac{(1+ T_j)^{{\ell_\circ}-j-\lambda_j+1}}{2+T_j}\cdot m^\#_{ij},
\end{equation}
where 
\begin{align*}
m^\#_{ij}=&d_{\lambda_i}(i)d_{\lambda_j}(j)+\sum_{m>0}(-1)^md_{\lambda_i+m}(i)d_{\lambda_j}(j)\\
&+\sum_{\ell>0}\sum_{m\geq 0}(-1)^\ell\left(\binom{\ell+m-1}{m}+\binom{\ell+m}{m}\right)(-1)^md_{\lambda_i+\ell+m}(i)d_{\lambda_j-\ell}(j).
\end{align*}

Now we expand the operators $\dfrac{(1+ T_i)^{\ell_\circ-i-\lambda_{i}+1}}{2+ T_i}$ in powers of $T_i$. Recall that $s_i=\ell_\circ-i-\lambda_{i}+1$ for $i=1,\ldots,\ell_\circ$. In fact, one can write 
\begin{align*}
\dfrac{(1+ T_i)^{s_i}}{1+(1+ T_i)}&=(1+ T_i)^{s_i}\cdot\sum_{u_i\geq 0}^\infty(-(1+T_i))^{u_i}=\sum_{u_i\geq 0}^\infty(-1)^{u_i}(1+T_i)^{u_i+s_i}\\
&=\sum_{u_i\geq 0}^\infty (-1)^{u_i}\sum_{v_i\geq 0}\binom{u_i+s_i}{v_i}T_i^{v_i}
\end{align*}
We note that $s_i$ can be negative, and the last equality follows from the binomial series. Then $\mathrm{Pf}(M^\star)$ can be written as
\begin{align*}
\mathrm{Pf}(M^\star)=\sum_{u_i, v_i\geq 0}\left(\prod_{i=1}^{\ell_\circ}(-1)^{u_i}\binom{u_i+s_i}{v_i}\right)\mathrm{Pf}(M^\natural)
\end{align*}
by linearity where $M^\natural=(m^\natural_{ij})$ is a $\ell_\circ\times\ell_\circ$ skew-symmetric matrix with entries
\begin{equation}\label{Mij}
m^\natural_{ij}=T_i^{v_i}T_j^{v_j}m^\#_{ij}.
\end{equation}

Since the raising operator $T_i$ increasing the index of $d_{\lambda_i}(i)$, $T_id_{\lambda_i}(i)=d_{\lambda_i+1}(i)$, we obtain


\begin{align*}
m^\natural_{ij}&=(\theta')^{\lambda_i+\lambda_j+v_i+v_j}\cdot\left(\dfrac{1}{(\lambda_i+v_i)!(\lambda_j+v_j)!}+\sum_{m>0}(-1)^m\dfrac{(\theta')^m}{(\lambda_i+m+v_i)!(\lambda_j+v_j)!}\right.\\
&\left.+\sum_{\ell>0}\sum_{m\geq 0}(-1)^\ell\left(\binom{\ell+m-1}{m}+\binom{\ell+m}{m}\right)(-1)^m\dfrac{(\theta')^m}{(\lambda_i+\ell+m+v_i)!(\lambda_j-\ell+v_j)!}\right),\\
\end{align*}
with Lemma \ref{Chrn}.

Therefore the Euler characteristic is given by 
\begin{align*}
\chi\left(\mathcal{O}_{V^r_{\bold{a}}(P)}\right)=
\sum_{u_i, v_i\geq 0}\left(\prod_{i=1}^{\ell_\circ}(-1)^{u_i}\binom{u_i+s_i}{v_i}\right)\int_{\mathrm{Pic}^{2g-2}(\mathcal{\widetilde{C}})}\mathrm{Pf}(M^\natural).
\end{align*}
\noindent Let us expand $\mathrm{Pf}(M^\natural)$ in powers of $\theta'$. Indeed, the entries of $M^\natural$ can be expressed as 
\begin{equation}\label{ent}
m^\natural_{i,j}=(\theta')^{\lambda_i+\lambda_j+v_i+v_j}\cdot\sum_{f_{i,j}\geq 0}g_{f_{i,j}}^{i,j}(\theta')^{f_{i,j}}.
\end{equation}

Using the definition of Pfaffian as in \cite[Appendix D]{FP}, we completely expand $Pf(M^\natural)$ in the powers of $\theta'$ with entries \eqref{ent} of $M^\natural$.
In fact, for $\sigma\in S_{\ell_\circ}$, we consider $\hat{f}(\sigma)$ for the powers $\bold{f}:=\{f_{i,j}\}_{1\leq i<j\leq \ell_\circ}$ of $\theta'$ in $m^\natural_{i,j}$ such that $f_{ij}=f_{ji}$.
Since $\ell_\circ$ is even, the Pfaffian $\mathrm{Pf}(M^\natural)$ is given by
\begin{align*}
\mathrm{Pf}(M^\natural)&=\dfrac{(\theta')^{|\lambda|+|{\mathbf{v}}|}}{2^{\ell_\circ/2}(\ell_\circ/2)!}\sum_{\sigma\in S_{\ell_\circ}}\sum_{i\geq0}\sum_{\bold{f}\in A_i^\sigma}\mathrm{sgn}(\sigma)\prod_{j=1}^{\ell_\circ/2}g_{f_{\sigma(2j-1),\sigma(2j)}}^{\sigma(2j-1),\sigma(2j)}(\theta')^i.
\end{align*}
In particular, if $\ell_\circ$ is odd, the skew-symmetric matrix $M^\star=(m^\star_{ij})$ is augmented by 
\[
m^\star_{0,j}=\dfrac{(1+T_j)^{\ell_\circ-j-\lambda_j+1}}{(2+T_j)}d_{\lambda_j}(j)
\]
for $j=1,\ldots,\ell_\circ$. So, $M^\natural$ is augmented by 
\begin{equation*}
m^\natural_{0j}=(\theta')^{\lambda_j+v_j}\cdot\dfrac{1}{(\lambda_j+v_j)!}.
\end{equation*}
Then the Pfaffian $\mathrm{Pf}(M^\natural)$ of the augmented matrix $M^\natural$ becomes
\begin{align*}
\mathrm{Pf}(M^\natural)&=\dfrac{(\theta')^{|\lambda|+|{\mathbf{v}}|}}{2^{(\ell_\circ+1)/2}((\ell_\circ+1)/2)!}\sum_{\sigma\in S_{\ell_\circ+1}}\sum_{i\geq0}\sum_{\bold{f}\in A_i^\sigma}\mathrm{sgn}(\sigma)\prod_{j=1}^{(\ell_\circ+1)/2}g_{f_{\sigma(2j-1),\sigma(2j)}}^{\sigma(2j-1),\sigma(2j)}(\theta')^i.
\end{align*}
for $\sigma\in S_{\ell_\circ+1}$ permuting $\{0,1,\ldots, \ell_\circ\}$.

Replacing $\theta'$ by $2\xi$ and using the Poincar\'e formula $\int \xi^\kappa=\kappa!$ conclude the theorem.
\end{proof}

\begin{conflict}

The author states that there is no conflict of interest.
\end{conflict}

\begin{data availability}

Data sharing not applicable to this article as no datasets were generated or analyzed during the current study.
\end{data availability}

\begin{acknowledgments}
We are grateful to David Anderson for his encouragement, invaluable comments and suggestions. We wishes to thank William Graham for helpful discussions. We also would like to express our gratitude to the Department of Mathematics at University of Georgia for various supports. We also greatly thank the anonymous referee for many constructible suggestions and comments which improve this manuscript significantly. 
\end{acknowledgments}

\bibliographystyle{amsplain}


\end{document}